\documentclass[12pt]{amsart} 
\usepackage{amsthm,amsbsy,amsfonts,amssymb,amscd}

\usepackage[mathscr]{euscript} 

\usepackage[T1]{fontenc} 
\usepackage{sidenotes} 

\usepackage{enumerate} 
\usepackage{tikz} 
\usetikzlibrary{calc,intersections,through, backgrounds,arrows,decorations.pathmorphing,backgrounds,positioning,fit,petri} 
\usetikzlibrary{calc}

\DeclareMathAlphabet{\mathpzc}{OT1}{pzc}{m}{it}

\usepackage{tikz-cd} 

\usepackage{hyperref}

\usepackage[margin=1in]{geometry}

\newcommand{\rad}{{\rm rad}} 

\newcommand{\soc}{{\rm soc}}

\newcommand{\cO}{\mathcal O}

\newcommand{\im}{{\rm{im}}}

\newcommand{\Hom}{\mathrm{Hom}}

\newcommand{\End}{\mathrm{End}}

\newcommand{\fg}{\mathfrak{g}}

\newcommand{\dfs}{{/\kern-2pt/}}

\newcommand{\mZ}{\mathbb{Z}}

\newtheoremstyle{notes} {} {} {} {} {\bfseries} {.} {.5em} {}

\theoremstyle{plain}

\newtheorem{prop}{Proposition}

\newtheorem{lemma}{Lemma}

\newtheorem{thm}{Theorem}

\theoremstyle{remark}

\newtheorem{rem}{Remark}

\newcommand{\fh}{\mathfrak{h}}
\swapnumbers

\newcommand{\lole}{\mathpzc{ll}}
\newcommand{\grle}{\mathpzc{gl}}

\theoremstyle{remark}

\newtheoremstyle{construction} {} {} {} {} {\bfseries} { } {0pt} {}

\theoremstyle{construction}

\title[Rigidity of tilting modules]{Rigidity of tilting modules in category $\cO$}

\author{Kevin Coulembier}

\keywords{}

\subjclass[2010]{}

\begin{document} 
\date{} 
\begin{abstract}In this note we give an overview of rigidity properties and Loewy length of tilting modules in the BGG category $\cO$ associated to a reductive Lie algebra. These results are well-known by several specialists, but seem difficult to find in the existing literature.
	\end{abstract}

\maketitle 



\section*{Introduction}
In highest weight categories, tilting modules are the modules which have both a standard and costandard filtration, see \cite{Ringel}. In category $\cO$ specifically, they are the self-dual modules with Verma filtration, see \cite{CI, ES}.

In each of the three `classical' Lie-theoretic settings (BGG category $\cO$, quantum groups at roots of unity, modular representation theory of reductive groups), these modules appear with their own applications, see for instance \cite{MaTilt, MO} for category $\cO$, \cite{AK} for quantum groups and \cite{RW} for modular representation theory. The level of understanding of the tilting modules seems to decrease along the above order.

The aim of this note is to write down some results which are known in category $\cO$, but perhaps not written out in full yet, and might be useful for studying the other settings.

\section{The results}In this section we state the main results, the proofs will follow in Section~\ref{SecProofs} and any unexplained notation will be introduced in Section~\ref{SecNot}.

A filtration of a module is called semisimple if all subquotients are semisimple. A {\em Loewy filtration} of a module $M$ is a semisimple filtration of minimal length (we will only consider modules where this is finite). That minimal length is by definition the {\em Loewy length}, $\lole(M)$. A module is {\em rigid} if it only has one Loewy filtration.

In the following theorem, the equivalence of (5) and (6) is a result of Stroppel in~\cite{Stroppel2} and the equivalence of (5), (7) and (8) (in the regular case) is a result of Irving in~\cite{Irving}.
\begin{thm}\label{Thm1}
For any integral dominant $\lambda\in\fh^\ast$ and $x\in W$ with~$y:=w_0x$, the following statements are equivalent:
\begin{enumerate}
\item $T(x\cdot\lambda)$ is rigid;
\item $T(x\cdot\lambda)$ has simple socle;
\item $\End_{\cO}(T(x\cdot\lambda))$ is commutative;
\item The (dual) Verma flag of $T(x\cdot\lambda)$ is multiplicity free;
\item $(P(y\cdot\lambda):\Delta(\lambda))=[\Delta(\lambda):L(y\cdot\lambda)]=1$;
\item $\End_{\cO}(P(y\cdot\lambda))$ is commutative;
\item $P(y\cdot\lambda)$ has simple socle;
\item $P(y\cdot\lambda)$ is rigid.
\end{enumerate}
\end{thm}

\begin{rem}
The proof of the equivalences
$$(2)\Leftrightarrow(3)\Leftrightarrow(4)\Leftrightarrow(5)\Leftrightarrow(6)\Leftrightarrow(7)$$
in Theorem~\ref{Thm1} does not use the Kazhdan-Lusztig conjecture, nor any results on Koszulity. Irving's proof that (8) implies (2)-(7) relies crucially on the validity of the Kazhdan-Lusztig conjecture. To prove that (1) implies (2)-(7), we furthermore use the existence of a positive grading (the Koszul grading) on $\cO$. For the proof that (2)-(7) imply (1) and (8) we apply Koszulity. This can probably be avoided however, using arguments as in the proof of \cite[Corollary 7]{Irving}.
\end{rem}

It is known, see the following lemma, that $T(x\cdot\lambda)$ is actually a submodule of $P(y\cdot\lambda)$. Theorem~\ref{Thm1} thus implies that (non-)rigidity is directly inherited by these submodules. Note that in general, a submodule of a (non-)rigid module need not be (non-)rigid.
\begin{lemma}\label{LemTrace}
With notation as in Theorem~\ref{Thm1}, the module $T(x\cdot\lambda)$ is the trace of $P(w_0\cdot\lambda)$ in $P(y\cdot\lambda)$. In other words, with  $H:=\Hom_{\cO}(P(w_0\cdot\lambda),P(y\cdot\lambda))$, we have
$$T(x\cdot\lambda)\;\cong\;\cup_{f\in H}\im(f).$$
\end{lemma}

Blocks in category $\cO$ are equivalent to finite dimensional Koszul algebras, see \cite{BGS}. Hence, the Loewy length $\lole$ is bounded by the graded length $\grle$, see Section~\ref{SecSocRad}. In the case of tilting modules we can say more.

\begin{prop}\label{PropLL}
For any tilting module in $\cO$, the graded and Loewy length coincide. Concretely, for any integral dominant $\lambda\in\fh^\ast$ and $x$ with maximal length in $\{y\in W\,|\,y\cdot\lambda=x\cdot\lambda\}$, we have
$$\lole( T(x\cdot\lambda))\;=\;\grle ( T(x\cdot\lambda))\;=\; 2\ell(w_0x)+1.$$
\end{prop}

In \cite{Hazi}, it is proved that tilting modules for quantum groups possess a `balanced' Loewy filtration, see Section~\ref{SecSocRad}. Furthermore, a very elegant algorithm to determine the layers of this filtration is given, although of course it still relies on Kazhdan-Lusztig combinatorics. The existence of a Koszul grading on $\cO$ along with some well known facts concerning tilting modules, imply that the analogue of that algorithm always yields the grading filtration (which is balanced by the self-duality of tilting modules) in category $\cO$.
\begin{lemma}\label{LemHazi} The following algorithm of \cite{Hazi} yields the grading filtration for tilting modules in category $\cO$.
\begin{enumerate}[a.]
\item Write the grading filtration of the Verma module $\Delta(\lambda)$. View this as a partial Loewy series for $T(\lambda)$ (namely the bottom layers). We will reflect Loewy layers about the ``middle'' Loewy layer in which $L(\lambda)$ appears.

\item Pick the highest ``unbalanced'' weight; that is, the largest $\mu<\lambda$ such that $L(\mu)$ appears below $L(\lambda)$ but there is no corresponding factor $L(\mu)$ in the reflected layer above $L(\lambda)$. \label{item:restart}

\item Add the grading filtration of $\Delta(\mu)$ to the partial Loewy series so that the head of $\Delta(\mu)$ is in the reflected Loewy layer above $L(\lambda)$.

\item Repeat from Step \ref{item:restart} until the Loewy series is balanced.
\end{enumerate}

\end{lemma}

\section{Category $\cO$, Ringel duality and Koszul grading}\label{SecNot}
We work over the ground field $\mathbb{C}$ of complex numbers.

\subsection{The Lie algebra}Let $\mathfrak{g}$ be a {\em reductive Lie algebra} with a fixed {\em triangular decomposition}
\begin{equation}\label{eq1}
\mathfrak{g}\;=\;\mathfrak{n}^-\oplus\mathfrak{h}\oplus\mathfrak{n}^+.
\end{equation}
Here $\mathfrak{h}$ is a fixed {\em Cartan subalgebra} and
$\mathfrak{b}=\mathfrak{h}\oplus\mathfrak{n}^+$ is the corresponding {\em Borel subalgebra}.
We denote by $W$ the associated Weyl group with longest element $w_0$. 
The half of the sum of all positive roots is denoted by $\rho\in\mathfrak{h}^\ast$
and the $W$-invariant form on $\mathfrak{h}^*$ is denoted $\langle \cdot,\cdot\rangle$.
The dot action of $W$ on $\fh^\ast$ is denoted by $w\cdot\lambda=w(\lambda+\rho)-\rho$. The partial order $\le$ on $W$ is the Bruhat order, see \cite[Section~0.4]{Humphreys}, with convention that the identity element $e\in W$ is minimal.

Let~$\Lambda\subset\mathfrak{h}^\ast$ denote the set of
{\em integral weights}, that is weights which appear in finite dimensional
$\mathfrak{g}$-modules.  The {\em dominant weights} form the subset
\begin{displaymath}
\Lambda^+=\{\lambda\,|\,\langle \lambda+\rho,\alpha\rangle \ge 0,\;\;\mbox{for all $\alpha\in \Delta^+$}\}.
\end{displaymath}
For $\lambda\in \Lambda^+$, we denote by $W_\lambda\subset W$ its stabiliser subgroup under the
dot action and by $w_0^\lambda$ the longest element in~$W_\lambda$. The set of longest
representatives in~$W_\lambda\backslash W$ is denoted by $X_\lambda$.

\subsection{Category $\cO$}

Consider the {\em BGG category  $\mathcal{O}$} associated to the triangular decomposition~\eqref{eq1},
see \cite{BGG, Humphreys}. Simple objects in~$\mathcal{O}$ are, up to isomorphism,
{\em simple highest weight modules} $L(\mu)$, where $\mu\in\mathfrak{h}^\ast$. The module
$L(\mu)$ is the simple top of the {\em Verma module}
$\Delta(\mu)$ and has highest weight~$\mu$. The projective cover of $L(\mu)$ in~$\mathcal{O}$
is denoted $P(\mu)$. The injective envelope of $L(\mu)$ in~$\mathcal{O}$
is denoted $I(\mu)$. 

Category $\cO$ has the simple preserving duality functor $\vee$ of \cite[Section~3.2]{Humphreys}.
The dual Verma module with socle $L(\mu)$ is $\nabla(\mu):=\Delta(\mu)^\vee$. The exact subcategory of $\cO$ of modules with a {\em Verma flag} (or standard filtration), see~\cite[Section~3.7]{Humphreys}, is denoted by $\cO^\Delta$. We also have the category $\cO^{\nabla}$ of modules with dual Verma flag. 
By \cite[Theorem~3.7]{Humphreys}, we have
\begin{equation}\label{eqQH}(M:\nabla(\mu))\;=\;\dim\Hom_{\cO}(\Delta(\mu),M),\end{equation}
for any $M\in \cO^{\nabla}$ and $\mu\in\fh^\ast$.
We also have $P(\mu)\in\cO^{\Delta}$ and the BGG reciprocity relation
$$(P(\mu):\Delta(\nu))=[\Delta(\nu):L(\mu)],$$
for all $\mu,\nu\in\fh^\ast$, see \cite[Theorem~3.11]{Humphreys}.

We will only consider the {\em integral part} $\mathcal{O}_\Lambda$ of $\mathcal{O}$ which contains all modules
with weights in~$\Lambda$. This is justified by \cite[Theorem~11]{SoergelD}.
The category $\mathcal{O}_\Lambda$ decomposes into {\em indecomposable} blocks
as follows:
\begin{displaymath}
\mathcal{O}_\Lambda\;=\;\bigoplus_{\lambda\in\Lambda^+}\mathcal{O}_\lambda,
\end{displaymath}
where $\mathcal{O}_\lambda$, for $\lambda\in\Lambda^+$, is the Serre subcategory of $\mathcal{O}$ generated by
all simples of the form~$L(x\cdot\lambda)$, where $x\in X_\lambda$. By \cite[Theorem~5.1]{Humphreys}, we have for all $\lambda\in \Lambda^+$ and $x,y\in X_\lambda$
\begin{equation}\label{eqBGG}
[\Delta(x\cdot\lambda):L(y\cdot\lambda)]\not=0\quad\Leftrightarrow\quad x\le y\quad\Leftrightarrow\quad \Delta(y\cdot\lambda)\subset \Delta(x\cdot\lambda).
\end{equation}
By \cite[Section~4.1]{Humphreys}, the socle of $\Delta(x\cdot\lambda)$, for any $x\in X_\lambda$, is $L(w_0\cdot\lambda)$. In particular, the socle of any module in $\cO^\Delta_\lambda$ is a direct sum of modules isomorphic to $L(w_0\cdot\lambda)$.

Let $A_\lambda$ be the endomorphism algebra of the projective generator 
$$P_\lambda:=\bigoplus_{x\in X_\lambda}P(x\cdot\lambda).$$ Then we have an equivalence of categories
$$\Hom_{\cO_\lambda}(P_\lambda,-):\;\;\cO_\lambda\;\tilde\to\;A_\lambda\mbox{-mod}.$$
We will use the same notation for the $\fg$-module $M\in \cO_\lambda$ and the $A_\lambda$-module $\Hom_{\cO_\lambda}(P_\lambda,M)$.

\subsection{Tilting modules}
For each $\mu\in \Lambda$, we have a unique indecomposable $\vee$-self-dual module $T(\mu)$ with a short exact sequence
\begin{equation}\label{DefT}0\to \Delta(\mu)\to T(\mu)\to K\to 0,\qquad\mbox{for some $K\in\cO^\Delta$,}\end{equation}
see~\cite[Chapter~11]{Humphreys}. We refer to~$T(\mu)$ as the {\em tilting module} corresponding to~$\mu$. The weight~$\mu$ is also the highest weight for which $T(\mu)$ has a non-zero weight space and the weight space~$T(\mu)_\mu$ has dimension one. 

The endomorphism algebra of 
$$T_\lambda:=\bigoplus_{x\in X_\lambda} T(x\cdot\lambda)$$ is known as the {\em Ringel dual algebra} of $A_\lambda$. It follows from \cite{SoergelT} or \cite{prinjective} that $A_\lambda$ is in fact Ringel self-dual\footnote{Note that the Ringel self-duality in~\cite{SoergelT} is somewhat hidden, as it gives Ringel duality between $A_{\lambda}$ and $A_{-(w_0(\lambda+\rho)+\rho)}$. However, since we have $W_{-w_0(\lambda+\rho)-\rho}=w_0W_\lambda w_0$, it follows from \cite[Theorem~11]{SoergelD} that $A_\lambda$ and $A_{-(w_0(\lambda+\rho)+\rho)}$ are actually isomorphic.}. As a consequence, we have the following lemma.

\begin{lemma}\label{LemRingel}
For every $\lambda\in \Lambda^+$ and $x,y\in X_\lambda$, we have
\begin{enumerate}[(i)]
\item $\End_{\cO}(T(x\cdot\lambda))\;\cong\;\End_{\cO}(P(w_0xw_0^\lambda\cdot\lambda))$;
\item $(T(x\cdot\lambda):\Delta(y\cdot\lambda))=(T(x\cdot\lambda):\nabla(y\cdot\lambda))=(P(w_0xw_0^\lambda \cdot\lambda):\Delta(w_0yw_0^\lambda\cdot\lambda))$.
\end{enumerate}
\end{lemma}
\begin{proof}
The equation $(T(x\cdot\lambda):\Delta(y\cdot\lambda))=(T(x\cdot\lambda):\nabla(y\cdot\lambda))$ is an immediate consequence of the self-duality of tilting modules under $\vee$. 
The other equalities follow from general principles of Ringel duality in~\cite[Section~6]{Ringel} applied to the case $\cO_\lambda$ in~\cite{SoergelT, prinjective}.
The combinatorics of the corresponding equivalence of categories (where we identify $A_\lambda$-mod and $\cO_\lambda$)
$$\Hom_{\cO_\lambda}(T_\lambda,-):\cO_\lambda^{\nabla}\to\cO_\lambda^{\Delta},$$ which by construction maps tilting modules to projective modules, is for instance worked out in~\cite[Theorem~8.1]{dualities}.

Note that we can combine Soergel's combinatorial functor $\mathbb{V}_\lambda=\Hom_{\cO_\lambda}(P(w_0\cdot\lambda),-)$, \cite[Struktursatz~9]{SoergelD} and Lemma~\ref{LemTrace} to give another proof of the isomorphism between the algebras $\End_{\cO}(T(x\cdot\lambda))$ and $\End_{\cO}(P(y\cdot\lambda))$.
\end{proof}

\subsection{Socle, radical and grading filtration}\label{SecSocRad}We review the two extremal Loewy filtrations, see also~\cite[Section~1.2]{Irving}. The {\em socle filtration},
$$0=\soc_0(M)\subset \soc_1(M)\subset \soc_2(M)\subset\cdots\subset M,$$
is the filtration where $\soc_k(M)$ is the unique submodule of $M$ such that $\soc_k(M)/\soc_{k-1}(M)$ is the socle of $M/\soc_{k-1}(M)$.
The {\em radical filtration},
$$0\subset \cdots\subset \rad^2(M)\subset\rad^1(M)\subset \rad^0{M}=M,$$
is the filtration where $\rad^i(M)$ is the radical of $\rad^{i-1}(M)$.

With $d=\lole(M)$, for any Loewy filtration
$$0=F_0M\subset F_1M\subset\cdots F_{d-1}M\subset F_dM=M,\quad\;\;\mbox{we have}\qquad\rad^{d-i}(M)\subseteq F_iM\subseteq\soc_i(M).$$
Clearly, a module is rigid if and only if the socle and radical filtration coincide.

A Loewy filtration $F_\bullet M$ of a module $M$, with~$d=\lole(M)$, is {\em balanced} if
$$F_iM/F_{i-1}M\;\cong\; F_{d-i+1}M/F_{d-i}M,\qquad \mbox{for all $1\le i\le d$.}$$

A finite dimensional algebra $A$ is `positively graded' if it has a grading $A=\bigoplus_{i\ge 0}A_i$, with~$A_0$ semisimple. Consider a finite dimensional graded module $M=\bigoplus_{i\in\mZ}M_i$. Let $k$ be the minimal degree for which $M_{k}$ is non-zero. Then the filtration of $M$ with 
$$F_iM=\bigoplus_{j< i+k}M_j,$$
is clearly a semisimple filtration and known as the {\em grading filtration}. The length of this filtration is the {\em graded length} $\grle(M)$. By definition, $\lole(M)\le \grle(M)$.

\subsection{The Koszul grading}
In \cite[Theorem~1.1.3]{BGS}, it is proved that $A_\lambda$ has a Koszul grading. 
The following lemma will therefore be relevant.
\begin{lemma}\cite[Proposition~2.4.1]{BGS}\label{LemBGS}
If a graded module $M$ over a Koszul algebra $A$ has simple socle (resp. simple top), then the socle filtration (resp. radical filtration) of $M$ coincides with the grading filtration.
In those cases, $\lole(M)=\grle(M)$.

Consequently, if $M$ has both simple top and socle, it is rigid.
\end{lemma}

 For a $\mZ$-graded vector space~$V$ and $i\in\mZ$, we denote by $V\langle i\rangle$ the graded vector space which is equal to $V$ as an ungraded space, with grading
$$V\langle i\rangle_j= V_{j-i},\quad\mbox{for all $j\in\mZ$}.$$

Projective and (dual) Verma modules admit graded lifts, see e.g.~\cite{BGS}. Furthermore, it is well-known (see e.g.~\cite{SHPO4}) that
\begin{equation}\label{glDelta}\lole (\Delta(x\cdot\lambda))=\grle (\Delta(x\cdot\lambda))=\ell(w_0x)+1,\quad\mbox{for all $x\in X_\lambda$.}\end{equation}

It is proved in~\cite{MO} that the tilting modules admit graded lifts, such that $\Delta(\mu)\hookrightarrow T(\mu)$ in \eqref{DefT} preserves the grading. We choose the normalisation of grading such that $\Delta(\mu)$ has its top in degree $0$. Moreover, we have the following lemma.
\begin{lemma}\label{LemFiltT}
Consider $\lambda\in\Lambda^+$ and $x\in X_\lambda$. Every subquotient of any standard filtration of the graded module $T(x\cdot\lambda)$,
which is not isomorphic to~$\Delta(x\cdot\lambda)$, has the form $\Delta(y\cdot\lambda)\langle l\rangle$ with~$l<0$ and $y\in X_\lambda$ with~$y<x$.
\end{lemma}
\begin{proof}
The condition $y<x$ follows from the combination of equation~\eqref{eqBGG} and Lemma~\ref{LemRingel}(ii). The case $\lambda=0$ is \cite[Lemma~2.4]{MaTilt}.
The singular case can be obtained from the regular case by the combinatorics of graded (exact) translation out of the wall, see e.g.~\cite{Stroppel}, which implies
$$\theta_\lambda^{out}T(x\cdot\lambda)=T(xw_0^\lambda\cdot0)\langle \ell(w_0^\lambda)\rangle,
$$
see e.g.~\cite[equation~(8)]{SHPO4} and
\begin{displaymath}
\left(\theta_{\lambda}^{out}\Delta(y\cdot\lambda):
\Delta(yu\cdot0)\langle j\rangle\right)=\delta_{j,l(u)}\qquad \text{ for all }\quad u\in W_\lambda, 
\end{displaymath}
with no other standard modules appearing in the filtration, see e.g.~\cite[Theorem~4.4]{dualities}.
\end{proof}

\section{Proofs}\label{SecProofs}

\begin{proof}[Proof of Theorem~\ref{Thm1}]
We choose $x\in X_\lambda$ and set $y:=w_0xw_0^\lambda\in X_\lambda$.
From Lemma~\ref{LemBGS}, we find (2)$\Rightarrow$(1) and (7)$\Rightarrow$(8). Note that, by self-duality, tilting modules have simple top if and only if they have simple socle. By definition projective covers have simple top.

The equivalence (3)$\Leftrightarrow$(6) follows from Lemma~\ref{LemRingel}(i).

The equivalence (5)$\Leftrightarrow$(6) is precisely \cite[Theorem~7.1]{Stroppel2}.

Now we prove (4)$\Leftrightarrow$(5). First we observe that $[\Delta(\lambda):L(y\cdot\lambda)]=1$ implies that $[\Delta(z.\lambda):L(y\cdot\lambda)]\le 1$ for all $z\in X_\lambda$, as follows from $\Delta(z\cdot\lambda)\subset\Delta(w_0^\lambda\cdot\lambda)=\Delta(\lambda)$, see \eqref{eqBGG}. By BGG reciprocity, (5) is thus equivalent to
$$(P(y\cdot\lambda):\Delta(z\cdot\lambda))\le 1,\quad\mbox{for all $z\in X_\lambda$.}$$ By Lemma~\ref{LemRingel}(ii), the above is in turn equivalent to (4).

Now we prove (2)$\Leftrightarrow$(4). By the above paragraph, we know that (4) is equivalent to $(T(x\cdot\lambda):\nabla(w_0\cdot\lambda))=1$. Since $T(x\cdot\lambda)\in\cO^{\Delta}_\lambda$, its socle consists of copies of $L(w_0\cdot\lambda)$. Furthermore, since $\Delta(w_0\cdot\lambda)=L(w_0\cdot\lambda)$, equation~\eqref{eqQH} implies that
$$\dim\Hom_{\cO}(L(w_0\cdot\lambda),T(x\cdot\lambda))=(T(x\cdot\lambda):\nabla(w_0\cdot\lambda)).$$
This shows that (2) and (4) are equivalent.

Now we prove (7)$\Rightarrow$(5). Assume that $(P(y\cdot\lambda):\Delta(\lambda))=r>1$. It follows immediately from \cite[Theorem~6.5]{Humphreys} that we have a monomorphism
$$\Delta(\lambda)^{\oplus r}\hookrightarrow P(y\cdot\lambda),$$
with cokernel in $\cO^{\Delta}$. In particular, $L(w_0\cdot\lambda)^{\oplus r}$ appears in the socle.

To prove that (8)$\Rightarrow$(5)$\Rightarrow$(7), we can just repeat the last two paragraphs of the proof of \cite[Corollary~7]{Irving}. 

Finally we prove (1)$\Rightarrow$(2). By Lemma~\ref{LemFiltT} and equation~\eqref{glDelta}, we find that the first non-zero module in the grading filtration of $T(x\cdot\lambda)$ is $L(w_0\cdot\lambda)$ (as the socle of $\Delta(x\cdot\lambda)$) in degree~$\ell(w_0x)$. If the socle of $T(x\cdot\lambda)$ is not simple, the grading filtration and socle filtration thus differ, which concludes the proof.
\end{proof}

\begin{proof}[Proof of Lemma~\ref{LemTrace}]
Note that $L(w_0\cdot0)\cong T(w_0\cdot0)$ is indeed the trace of $P(w_0\cdot0)$ in $\Delta(w_0\cdot0)$. For any $w\in W$, denote by $\theta_w$ the unique projective functor, see \cite{BG}, which maps $\Delta(0)\cong P(0)$ to $P(w\cdot0)$. From the fact that $\theta_w$ is left and right adjoint to $\theta_{w^{-1}}$ and the fact that $\theta_w P(w_0\cdot 0)$ is a direct sum of copies of $P(w_0\cdot0)$ it follows immediately that, if $K$ is the trace of $P(w_0\cdot0)$ in $M$, then $\theta_wK$ is the trace of $P(w_0\cdot0)$ in $\theta_wM$. 

We thus find that $\theta_wL(w_0\cdot 0)$ is the trace of $P(w_0\cdot0)$ in $P(x\cdot0)$. The former is precisely $T(w_0w\cdot0)$, see e.g. \cite[Proposition~5.10]{dualities}. The statement for arbitrary integral dominant $\lambda$ follows similarly from the above and translation onto the wall.
\end{proof}

\begin{proof}[Proof of Proposition~\ref{PropLL}]
We have
$$\lole (T(x\cdot\lambda))\le\grle (T(x\cdot\lambda))=2\ell(w_0x)+1,$$
see e.g. \cite{SHPO4}.
We will trace the unique simple constituent $L=L(x\cdot\lambda)$ in the semisimple filtrations of~$T=T(x\cdot\lambda)$. Consider an arbitrary semisimple filtration
$$0=T_0\subset T_1\subset T_2\subset \cdots\subset T_{k-1}\subset T_k=T.$$
Since $L$ is the top of the submodule $\Delta(x\cdot\lambda)\subset T$, see~\eqref{DefT}, this means that the $j$ for which $L$ is a submodule of $T_j/T_{j-1}$ satisfies
$$j\;\ge\;\lole\Delta(x\cdot\lambda)\;=\; \ell(w_0x)+1. $$
Applying the duality functor $\vee$ to the filtration of $T=T^\vee$ yields a semisimple filtration 
$$0=T'_0\subset T'_1\subset T'_2\subset \cdots\subset T'_{k-1}\subset T'_k=T,$$
with~$T'_i=(T/T_{k-i})^{\vee}$.
We can thus also conclude that
$$k-j+1\;\ge\; \ell(w_0x)+1.$$
In conclusion, we find
$k\ge 2\ell(w_0x)+1,$
which implies $\lole (T(x\cdot\lambda))\ge 2\ell(w_0x)+1$.
\end{proof}

\begin{proof}[Proof of Lemma~\ref{LemHazi}]
This follows immediately from Lemma~\ref{LemFiltT}, equation~\eqref{eqBGG} and (graded) self-duality of tilting modules.
\end{proof}

\subsection*{Acknowledgement}
The author thanks Jim Humphreys for discussions, pointing out the gap in the literature concerning rigidity of tilting modules in $\cO$ and bringing the paper \cite{Hazi} to his attention.

\end{document}